\documentclass[11pt,leqno]{article}
\usepackage{graphicx, amsfonts, amsthm, amsxtra, amssymb, verbatim, makeidx}
\usepackage[mathscr]{euscript}
\usepackage{hyperref}
\makeindex
\makeglossary
\begin{document}
\newtheorem{theo}{Theorem}
\newtheorem{exam}{Example}
\newtheorem{coro}{Corollary}
\newtheorem{defi}{Definition}
\newtheorem{prob}{Problem}
\newtheorem{lemm}{Lemma}
\newtheorem{prop}{Proposition}
\newtheorem{rem}{Remark}
\newtheorem{conj}{Conjecture}
\newtheorem{calc}{}
\newtheorem{property}{Property}
\newtheorem{hypo}{Hypothesis}


\def\Z{\mathbb{Z}}                   
\def\Q{\mathbb{Q}}                   
\def\C{\mathbb{C}}                   
\def\N{\mathbb{N}}                   
\def\uhp{{\mathbb H}}                
\def\A{\mathbb{A}}                   
\def\dR{{\rm dR}}                    
\def\F{{\cal F}}                     
\def\Sp{{\rm Sp}}                    
\def\Gm{\mathbb{G}_m}                 
\def\Ga{\mathbb{G}_a}                 
\def\Tr{{\rm Tr}}                      
\def\tr{{{\mathsf t}{\mathsf r}}}                 
\def\spec{{\rm Spec}}            
\def\ker{{\rm ker}}              
\def\GL{{\rm GL}}                
\def\k{{\sf k}}                     
\def\ring{{\sf R}}                   
\def\X{{\sf X}}                      
\def\T{{\sf T}}                      
\def\Ts{{\sf S}}
\def\cmv{{\sf M}}                    
\def\BG{{\sf G}}                       
\def\podu{{\sf pd}}                   
\def\ped{{\sf U}}                    
\def\per{{\sf  P}}                   
\def\gm{{\sf  A}}                    
\def\gma{{\sf  B}}                   
\def\ben{{\sf b}}                    

\def\Rav{{\mathfrak M }}                     
\def\Ram{{\mathfrak C}}                     
\def\Rap{{\mathfrak G}}                     

\def\nov{{\sf  n}}                    
\def\mov{{\sf  m}}                    
\def\Yuk{{\sf Y}}                     
\def\Ra{{\sf R}}                      
\def\hn{{\sf h}}                      
\def\cpe{{\sf C}}                     
\def\g{{\sf g}}                       
\def\t{{\sf t}}                       
\def\pedo{{\sf  \Pi}}                  

\def\Der{{\rm Der}}                   
\def\MMF{{\sf MF}}                    
\def\codim{{\rm codim}}                
\def\dim{{\rm    dim}}                
\def\Lie{{\rm Lie}}                   
\def\gg{{\mathfrak g}}                

\def\u{{\sf u}}                       

\def\imh{{  \Psi}}                 
\def\imc{{  \Phi }}                  
\def\stab{{\rm Stab }}               
\def\Vec{{\rm Vec}}                 
\def\prim{{\rm prim}}                  

\def\Fg{{\sf F}}     
\def\hol{{\rm hol}}  
\def\non{{\rm non}}  
\def\alg{{\rm alg}}  

\def\bcov{{\rm \O_\T}}       

\def\leaves{{\cal L}}        

\def\GM{{\rm GM}}

\def\perr{{\sf q}}        
\def\perdo{{\cal K}}   
\def\sfl{{\mathrm F}} 
\def\sp{{\mathbb S}}  

\newcommand\diff[1]{\frac{d #1}{dz}} 
\def\End{{\rm End}}              

\def\sing{{\rm Sing}}            
\def\cha{{\rm char}}             
\def\Gal{{\rm Gal}}              
\def\jacob{{\rm jacob}}          
\def\tjurina{{\rm tjurina}}      
\newcommand\Pn[1]{\mathbb{P}^{#1}}   
\def\Ff{\mathbb{F}}                  

\def\O{{\cal O}}                     
\def\as{\mathbb{U}}                  
\def\ring{{\mathsf R}}                         
\def\R{\mathbb{R}}                   

\newcommand\ep[1]{e^{\frac{2\pi i}{#1}}}
\newcommand\HH[2]{H^{#2}(#1)}        
\def\Mat{{\rm Mat}}              
\newcommand{\mat}[4]{
     \begin{pmatrix}
            #1 & #2 \\
            #3 & #4
       \end{pmatrix}
    }                                
\newcommand{\matt}[2]{
     \begin{pmatrix}                 
            #1   \\
            #2
       \end{pmatrix}
    }
\def\cl{{\rm cl}}                

\def\hc{{\mathsf H}}                 
\def\Hb{{\cal H}}                    
\def\pese{{\sf P}}                  

\def\PP{\tilde{\cal P}}              
\def\K{{\mathbb K}}                  

\def\M{{\cal M}}
\def\RR{{\cal R}}
\newcommand\Hi[1]{\mathbb{P}^{#1}_\infty}
\def\pt{\mathbb{C}[t]}               
\def\W{{\cal W}}                     
\def\gr{{\rm Gr}}                
\def\Im{{\rm Im}}                
\def\Re{{\rm Re}}                
\def\depth{{\rm depth}}
\newcommand\SL[2]{{\rm SL}(#1, #2)}    
\newcommand\PSL[2]{{\rm PSL}(#1, #2)}  
\def\Resi{{\rm Resi}}              

\def\L{{\cal L}}                     
\def\Aut{{\rm Aut}}              
\def\any{R}                          
\newcommand\ovl[1]{\overline{#1}}    

\newcommand\mf[2]{{M}^{#1}_{#2}}     
\newcommand\mfn[2]{{\tilde M}^{#1}_{#2}}     

\newcommand\bn[2]{\binom{#1}{#2}}    
\def\ja{{\rm j}}                 
\def\Sc{\mathsf{S}}                  
\newcommand\es[1]{g_{#1}}            
\newcommand\V{{\mathsf V}}           
\newcommand\WW{{\mathsf W}}          
\newcommand\Ss{{\cal O}}             
\def\rank{{\rm rank}}                
\def\Dif{{\cal D}}                   
\def\gcd{{\rm gcd}}                  
\def\zedi{{\rm ZD}}                  
\def\BM{{\mathsf H}}                 
\def\plf{{\sf pl}}                             
\def\sgn{{\rm sgn}}                      
\def\diag{{\rm diag}}                   
\def\hodge{{\rm Hodge}}
\def\HF{{\sf F}}                                
\def\WF{{\sf W}}                               
\def\HV{{\sf HV}}                                
\def\pol{{\rm pole}}                               
\def\bafi{{\sf r}}
\def\id{{\rm id}}                               
\def\gms{{\sf M}}                           
\def\Iso{{\rm Iso}}                           

\def\hl{{\rm L}}    
\def\imF{{\rm F}}
\def\imG{{\rm G}}


\def\HL{{\rm Ho}}     
\def\NLL{{\rm NL}}   

\def\RG{{\bf G}}     
\def\rg{{\bf g}}     
\def\rbullet{{\cdot}}

\def\Smat{{\sf S}}

\def\sinone{{W}}
\def\sintwo{{\tilde W}}

\begin{center}
{\LARGE\bf  Gauss-Manin connection in disguise: \\
Noether-Lefschetz and Hodge loci
\footnote{ 
Math. classification: 14N35, 
14J15, 32G20
}
}
\\
\vspace{.25in} {\large {\sc Hossein Movasati}}\footnote{
Instituto de Matem\'atica Pura e Aplicada, IMPA, Estrada Dona Castorina, 110, 22460-320, Rio de Janeiro, RJ, Brazil (Present Address: Mathematics Department, Harvard University),
{\tt www.impa.br/$\sim$ hossein, hossein@impa.br}}
\end{center}

\begin{abstract}
 We give a classification of components of the Hodge locus in any 
parameter space of smooth projective varieties. This is done using  
determinantal varieties constructed from the infinitesimal variation 
of Hodge structures (IVHS) of the underlying family. As a corollary we prove that
the minimum codimension for the components of the Hodge locus in the parameter
space of $\mov$-dimensional hypersurfaces of degree
$d$ with $d\geq 2+\frac{4}{\mov}$ and in a Zariski neighborhood of the point representing
the Fermat variety,  is obtained by the locus
of hypersurfaces passing through an $\frac{\mov}{2}$-dimensional 
linear projective space. 
In the particular case of surfaces in the projective space of dimension three,
this is a theorem of Green and Voisin. In this case our
classification under a computational hypothesis on IVHS
implies  a weaker version of the Harris-Voisin conjecture which says that 
the set of special components of the 
Noether-Lefschetz locus  is not Zariski dense in the parameter space.   
\end{abstract}
\tableofcontents

\section{Introduction}
The Hodge conjecture implies that the components of the Hodge locus in a parameter space of smooth projective varieties  
are algebraic, and moreover, they are  
defined over the algebraic closure of the base field. 
The algebraicity statement has been  successfully proved by Cattani, Deligne and Kaplan in \cite{cadeka} using transcendental methods in Hodge theory, 
and hence, their proof does not give any light into the second part of the above statement. 
The classification of the components of the Hodge locus according
to their codimension is another important challenge in Hodge theory. 
Here, the Hodge conjecture has not so much to say. A typical
evidence to this is the particular case of Noether-Lefschetz locus, where the Hodge conjecture is known as Lefschetz $(1,1)$ theorem.
In this case it was observed that there are two classes of components, general and special ones. Ciliberto, Harris and Miranda in 
\cite{CHM88}  proved that general components are dense 
in the parameter space in both usual and Zariski topology. Harris conjectured that special components must be  
finite and Voisin found counterexamples to this, see \cite{voisin89, voisin90, voisin1991}. 
She then  conjectured that special components are not Zariski dense. We refer to this as the 
Harris-Voisin conjecture.

In the present article,  we first put the Harris-Voisin conjecture in the general framework
of Hodge loci and then we give a conjectural 
description of a proper algebraic subset of the parameter space which contains 
components of the Hodge locus in a wide range of codimensions. 
Partial verifications of our conjecture in the case of hypersurfaces give
us the precise description of the component which acquires the minimum codimension.  
Our results are stated using the infinitesimal variation of Hodge structures (IVHS) 
invented by Carlson, Donagi, Griffiths, Green and 
Harris in order to 
avoid the transcendental nature of the variation of Hodge structures, 
see \cite{CGGH1983}. 
The mentioned authors used IVHS in single points to attack classical problems such as 
Torelli problem and Noether's Theorem. 

For the proof of our main results we have to go back to the 
origin of IVHS which is the 
algebraic Gauss-Manin connection of the corresponding family. 
From this we 
construct a holomorphic foliation in a larger parameter space whose leaves and singularities 
are responsible for the classification 
of the components of the Hodge locus. 
Despite the fact that the concept of a foliation/integrable distribution is quit 
old in differential geometry, their applications to classical 
problems in algebraic geometry  and Hodge theory, 
in the way we do, must be considered as our main
contribution to the literature. 
The construction of such foliations goes back to the works of the author on differential 
equations of modular
forms and their generalization to Calabi-Yau varieties, 
see \cite{ho06-1,ho18} and the references therein. 
For many examples such foliations are given by vector fields which
are natural generalizations of Darboux, Halphen and Ramanujan vector fields. By Gauss-Manin connection in disguise
we mean such vector fields and foliations.
The terminology arose from a private letter of Pierre Deligne to the author \cite{del-letter}. 
For a fast review of the results in Hodge locus the reader is 
referred to Voisin's expository article \cite{voisinHL}. 

\subsection{Main results}
Let $Y\to V$ be a family of smooth complex projective varieties and let $V$ be irreducible, smooth and affine. 
For an even number $\mov$, an irreducible component $H$ of the Hodge locus $\HL_\mov(Y/V)$ is  any  irreducible closed subvariety of 
$V$ with a continuous family of Hodge classes  $\delta_t \in H^{\mov}(Y_t,\Q)\cap H^{\frac{\mov}{2},\frac{\mov}{2}}$ in varieties $Y_t, t\in H$ 
such that for points $t$ in a Zariski open subset of $H$, the monodromy of $\delta_t$  to a point in a neighborhood  
(in the classical topology of $V$) of $t$  and outside $H$,  is no more a Hodge class. 
Let 
\begin{equation}
\label{IVHS-sept}
H^1(Y_t,T_{Y_t})_\theta\times   H^{\mov-k}(Y_t,\Omega_{Y_t}^{k}) \to   H^{\mov-k+1 }(Y_t,\Omega_{Y_t}^{k-1})
\end{equation}
be  the $k$-th infinitesimal variation 
of Hodge structures at $t$, IVHS for short, with $k=\frac{\mov}{2}+1$ 
(we will not need the intersection form of IVHS except in \S \ref{popped}).
Here,  $H^1(Y_t,T_{Y_t})_\theta\subset H^1(Y_t,T_{Y_t})$ corresponds to projective deformations of $Y_t$, see \S\ref{16aug2014-2}.
From \eqref{IVHS-sept} we derive
\begin{equation}
\label{IVHS-2014}
H^{\mov-k+1 }(Y_t,\Omega_{Y_t}^{k-1})^{*}\to {\rm Hom} \left (H^1(Y_t,T_{Y_t})_\theta,  
H^{\mov -k}(Y_t,\Omega_{Y_t}^{k})^*  \right )
\end{equation}
where $*$ means dual and we have removed the zero element from both sides of \eqref{IVHS-2014}. 
Let $D_{s,t},\ s\in\N_0$ be the determinantal subvariety of the right hand side of \eqref{IVHS-2014}
consisting of homomorphisms of rank $\leq s$. 
For simplicity,  throughout the text we assume that the Kodaira-Spencer map
\begin{equation}
 \label{arezoyematoyito}
(T_V)_t\to H_1(Y_t,T_{Y_t})_\theta
\end{equation}
is surjective, however, all the arguments are valid when we replace $H_1(Y_t,T_{Y_t})_\theta$ with 
the image of \eqref{arezoyematoyito}. 
\begin{theo}
\label{maintheo}
 If the image of \eqref{IVHS-2014} with $k=\frac{\mov}{2}+1$  does not intersect $D_{s,t}$ for some $s$ and $t\in V$ then
 there is a Zariski open neighborhood $U$ of $t$ in $V$ such that 
 all the components of the Hodge locus $\HL_\mov(Y/V)$ intersecting $U$ have codimension $\geq s+1$. 
\end{theo}
The above theorem is a direct consequence  of Voisin's results on  the Zariski tangent
space of the components of the Hodge locus, see \cite{vo03} Lamma 5.16, p. 146. For a proof see \S\ref{popped}. Stated in this  format it becomes directly related to a weaker version of 
Harris-Voisin's conjecture, and it seems to me that its power and importance has been neglected in the literature, see 
Theorem \ref{maintheo2} and Theorem \ref{maintheo3} below. In this article we give a new proof of Theorem \ref{maintheo} which is 
based on the construction of a larger parameter space $\T$, a modular foliation in $\T$ and the notion of Hodge locus
with constant periods in $\T$. Our proof is purely algebraic, whereas Voisin's proof is based on local 
analytic
study of the Hodge locus. This might open up a new point of view to the Cattani-Deligne-Kaplan 
theorem discussed
at the beginning of the present paper. An advantage of our proof is that it says which part of the Gauss-Manin connection is absent
in IVHS and it  is needed for a full solution of the Harris-Voisin conjecture. Another advantage 
is that it gives a precise
description of the polynomial equations for the periods of Hodge classes, 
see for instance the end of \S\ref{geryekardam}. 

Theorem  \ref{maintheo} for $s=0$ is the  classical Noether's theorem.
It says that if \eqref{IVHS-2014} is injective (or equivalently if the map \eqref{IVHS-sept} 
is surjective) then the components of the Hodge locus
in $V$ are proper analytic subsets of $V$. Since we know that the set of such 
components is enumerable, we conclude that for a generic $Y_t$, the $\mov$-dimensional
Hodge classes are complete intersection of $Y_t$ with another variety in the ambient projective space and hence are algebraic, 
see \cite{CGGH1983} page 71 and \cite{Harris85} page 56.  From now on we assume that \eqref{IVHS-2014}
is injective and so we can take the induced map after projectivization.  
In the case of hypersurfaces, IVHS 
can be computed explicitly and
a simple analysis of the hypothesis of Theorem \ref{maintheo} for the Fermat variety gives us:
\begin{theo}
\label{maintheo2}
 Let $V$ be the parameter space of smooth hypersurfaces  of degree $d$ in $\Pn {\mov+1}$ and let
 $0\in V$ be the parameter of  the Fermat variety. Assume that $d\geq 2+\frac{4}{\mov}$.
 There is a Zariski open neighborhood $U$ of $0\in V$ 
 such that 
 all the components of the Hodge locus $\HL_\mov(Y/V)$ intersecting $U$ have codimension $\geq  \bn{\frac{\mov}{2}+d}{d}-(\frac{\mov}{2}+1)^2$.
 The lower bound is obtained by the locus $H$ of hypersurfaces
 containing a linear projective space $\Pn {\frac{\mov}{2}}\subset\Pn {\mov+1}$. 
\end{theo}
We usually call $0\in V$ the Fermat point. 
The above theorem  for $\mov=2$ and $U=V$ was conjectured in \cite{CGGH1983}, I. 
It was independently proved by 
Green in \cite{green1988, green1989} and Voisin in \cite{voisin1991}. 
In this case it is also proved that
$H$ is the only component of codimension $d-3$. 
Note that for $\mov\geq 4$  the Hodge conjecture 
is not known and Theorem \ref{maintheo2} is independent of this. 
We may analyze  IVHS for other single points in the parameter
space of hypersurfaces or complete intersections and 
get further results similar to Theorem \ref{maintheo2}. 
One can also take the bundle of IVHS in \eqref{IVHS-2014} and try to compute
its first order approximation in single points. We do this around the Fermat point
and we get the following. Let 
\begin{equation}
\label{21oct2014}
I_N:=\left \{ (i_0,i_1,\ldots,i_{\mov+1})\in \Z^{\mov+2}\mid 0\leq i_e\leq d-2, \ \ i_0+i_1+\cdots+i_{\mov+1}=N\right\}
\end{equation}
for $N=0,1,2,\ldots, (d-2)(\mov+2)$ and consider independent variables 
$x_i$ indexed by $i\in I_{(\frac{\mov}{2}+1)d-\mov-2}$. 
For any other $i$ which is not in the set 
$I_{(\frac{\mov}{2}+1)d-\mov-2}$, $x_i$ by definition is zero.   
Let $M:=[x_{i+j}]$ be a matrix whose rows and columns
are indexed by $i\in I_{\frac{\mov}{2}d-\mov-2}$ and $j\in I_d$, respectively, and
in its $(i,j)$ entry we have $x_{i+j}$. The matrix $M$ is obtained
by IVHS for the Fermat point.
For $j, \alpha\in I_d$ and $i\in I_{\frac{\mov}{2}d-\mov-2}$ 
such that for a unique $0\leq \check e\leq \mov+1$ we have $i_{\check e}+j_{\check e}\geq d-1$, let us define
$$
i+_\alpha j=i+j+\alpha-(0,\cdots,0,\overbrace{d}^{\check e\text{-th place}},0,\cdots,0).
$$
For other pairs of $(i,j)$ let $i+_\alpha j=0\in\Z^{\mov+2}$ (it can be any element outside $I_{(\frac{\mov}{2}+1)d-\mov-2}$). 
We define the matrix
$\check M_\alpha$ in the following way. 
For $(i,j)$ as above and in the first case, the $(i,j)$ entry of $\check M_\alpha$ is
$\alpha_{\check e}\cdot x_{i+_\alpha j }$, and elsewhere entries  are zero. 
The matrix $N_{j,\alpha}$ is  obtained
by replacing the $j$-th column of $M$ with the $j$-th column of $\check M_\alpha$. 
We define the homogeneous ideal ${\cal I}_s^1\subset \C[x],\ \ s=0,1,2,\ldots$ to 
be generated by 
\begin{eqnarray}
\label{23oct2014-1}
 {\rm minor}_{s+1} (M), & & \\ \label{23oct2014-2}
 \sum_{j\in I_d} {\rm minor}_{s+1}( N_{j,\alpha}), & & \alpha\in I_d,  
\end{eqnarray}
where "minor" runs through all minors of 
$(s+1)\times (s+1)$ submatrices of a matrix. Note that once we fix
a block of $(s+1)\times (s+1)$ matrix for making a determinant, it is the same for
all the matrices $N_{j,\alpha}$ in the sum \eqref{23oct2014-2}. 
Let also ${\cal I}_s^0\subset \C[x]$ be the homogeneous ideal generated by \eqref{23oct2014-1}.
This is the ideal of $(s+1)\times (s+1)$ minors of $M$. 
We define $s_{\rm max}^i,\ \ i=0,1$ to be the maximum $s$ such that 
${\rm Zero}({\cal I}_s^i)=\{0\}$. 
The proof of Theorem \ref{maintheo}
is reduced to check the  equality 
\begin{equation}
\label{mytheory}
s_{\rm max}^0=
\bn{\frac{\mov}{2}+d}{d}-(\frac{\mov}{2}+1)^2-1,
\end{equation}
that is, 
if for some $x_i$'s 
the rank of $M$ is $\leq s_{\rm max}^0$
then all $x_i$'s are zero, see Proposition \ref{yademadaram}.
\begin{theo}
\label{maintheo3}
There is a Zariski open subset $U$ of the parameter space $V$ of smooth 
hypersurfaces in $\Pn {\mov+1}$ such that 
all the components of the Hodge locus 
$\HL_\mov(Y/V)$ intersecting $U$ have codimension 
$\geq s_{\rm max}^1+1$. 
\end{theo}
We were not able to compute $s_{\rm max}^1$ neither by hand nor
by computer. Some methods using both theoretical and computational 
aspects of ideals and their Gr\"obner basis seems to be necessary for computing $s_{\rm max}^1$.  

\subsection{Harris-Voisin conjecture}
In this section we assume that \eqref{IVHS-2014} is injective and hence, Noether's theorem
is valid for $Y\to V$.  
Let
\begin{eqnarray}
 \label{contradiction90}
a &:=&\dim H^{\frac{\mov}{2}-1}(Y_t,\Omega_{Y_t}^{ \frac{\mov}{2}+1})= 
\# I_{\frac{\mov}{2}d-\mov-2}, \\ \nonumber
b &:=& \dim H^{\frac{\mov}{2} }(Y_t,\Omega_{Y_t}^{\frac{\mov}{2}})= \# I_{(\frac{\mov}{2}+1)(d-2)},\\ \nonumber
r&:=&\dim  H^1(Y_t,T_{Y_t})_\theta= \# I_{d},\\ \nonumber
c &:=& \text{ the maximum rank of the image of \eqref{IVHS-2014} for generic $t$}
\end{eqnarray}
where $I_N$ is defined in \eqref{21oct2014}. 
The first equalities/definitions are for a general smooth 
projective variety $Y_t$ and the second equalities are for smooth hypersurfaces of degree $d$ in $\Pn {\mov+1}$. 
By a theorem of Voisin (see Proposition 5.14 \cite{vo03}) the codimension of the 
components of $\HL_\mov(Y/V)$ is $\leq a$.
The main challenge in front of us
is to find the maximum value of $s$  for a fixed or a  generic $t\in V$ such that the hypothesis of Theorem 
\ref{maintheo} is true.  The vector spaces in \eqref{IVHS-2014} are fibers of algebraic bundles over $V$. 
Let 
\begin{equation}
\label{IVHS-bundle}
{\mathcal H}^{b, *}\to  {\rm Hom} \left ({\mathcal H}^r, {\mathcal H}^{a,*}\right )
\end{equation}
be the bundle homomorphism 
obtained from \eqref{IVHS-2014}. We denote by $\sinone_{c-s}\subset {\mathcal H}^{b, *}$ 
the determinantal variety of homomorphisms of rank $\leq s$ of 
\eqref{IVHS-bundle} (the reason for this index notation is explained in \S\ref{hanoozhich}). 
It is the pull-back of the determinantal variety $D_s$ in ${\rm Hom} \left ({\mathcal H}^r, {\mathcal H}^{a,*}\right )$ by the map
\eqref{IVHS-bundle}. Two important numbers in our study are
$$
s_{\rm max}:= \text{ the maximum $s$ such that the projection 
$\sinone_{c-s}\to V$ is not dominant,} 
$$
\begin{equation}
\label{30sept2014}
\check s_{\rm max}:=a- \left \lceil \sqrt{ \left(\frac{r-a}{2}\right)^2+b}\ \ -\ \ \frac{r-a}{2}\right  \rceil,
\end{equation}
where for $x\in\R$
we define $\lceil x\rceil$ the unique
integer with $\lceil x\rceil-1< x \leq \lceil x\rceil$.
We have 
\begin{equation}
s_{\rm max}^0\leq 
s_{\rm max}^1 \leq 
s_{\rm max}\leq 
\check  s_{\rm max}.
\end{equation}
\begin{conj}
\label{paperyounes}
If the Kodaira-Spencer map \eqref{arezoyematoyito} is surjective then for a generic $t\in V$, the map \eqref{IVHS-2014} 
is transversal to the determinantal variety $D_{\check s_{\rm max},t }$ 
of homomorphisms of rank $\leq \check s_{\rm max}$ (and hence  does not intersect it).
\end{conj}
If this conjecture is true then $s_{\rm max}=\check s_{\rm max}$.
In order to explain the content of  Conjecture \ref{paperyounes} we 
consider the following case.
Let $V$ be the parameter space of smooth complex surfaces  of degree $d$ in $\Pn 3$ and let $Y/V$ be the corresponding 
family.  
For $t\in V$ let 
$f_t=f_t(X_0,X_1,X_2,X_{3})$ be 
the corresponding homogeneous 
polynomial and $Y_t:=\Pn {}\{f_t=0\}$. The map in \eqref{IVHS-sept} for $\mov=k=2$ 
is identified  by the multiplication 
of polynomials:
\begin{equation}
\label{20sep2014-1}
(\C[X]/J)_d\times (\C[X]/J)_{d-4}\to (\C[X]/J)_{2d-4},
\end{equation}
 where $J:= \jacob(f_t)$ is the Jacobian ideal of $f_t$.
In this case 
the Hodge locus $\NLL_d:=\HL_2(Y/V)$ is known as the Noether-Lefschetz locus and we have 
a good understanding of it, see \S\ref{NLlocus} for a review of some results. The components of $\NLL_d$ 
have codimension $\leq \bn{d-1}{3}$ and  Ciliberto, Harris and Miranda in \cite{CHM88} have constructed  infinite number of  components 
of codimension $a$ whose union is dense in $V$ with both usual and Zariski topology. 
A component of this codimension  is called a general component and all others are called special. Joe Harris conjectured that  there must be 
finitely many  special components and Voisin  
gave counterexamples, see \cite{voisin89, voisin90, voisin1991}. 
She then  formulated the conjecture below:
\begin{conj}
\label{17july2014}
(Harris-Voisin) The union of all special  components of the Noether-Lefschetz locus  is not Zariski dense in $V$. 
\end{conj}
We  have:
\begin{coro}
\label{estamuydirty}
If  Conjecture \ref{paperyounes} is true for \eqref{20sep2014-1} with $d\geq 4$ then
there is a Zariski open subset $U$ of $V$ such that all components 
of the Noether-Lefschetz locus
intersecting 
$U$ have codimension bigger than or equal $s_{\rm max}+1$, where
\begin{equation}
\label{9oct2014}
s_{\rm max}=\frac{(d-1)(d-2)(d-3)}{6}- \left\lceil \sqrt{ d^4+\frac{2}{3}d^3-16d^2+\frac{7}{3}d+48}-(d^2-7) \right \rceil 
\end{equation}
\end{coro}
Let $y_0$ be the quantity inside $\lceil\cdot \rceil$ in \eqref{9oct2014}. For any small real number $\epsilon$
we have
$$
\frac{1}{3}d-\frac{19}{18} \leq y_0\leq \frac{1}{3}d-\frac{19}{18}+\epsilon
$$
where the left hand side equality is for all $d\geq 4$ and the right hand side equality 
is for big $d$ depending on $\epsilon$. Therefore, once Conjecture 
\ref{paperyounes} is verified
we get a good approximation to the Harris-Voisin conjecture. The reason why IVHS cannot do better than this
will be explained during the proof of Theorem \ref{maintheo} and \S\ref{19aug2014}.

\subsection*{Organization of the text}
The organization of the article and the main ideas behind the proof
of our main theorems are as follows.
For a proof Theorem \ref{maintheo} using Voisin's results, 
the reader can go directly to
\S\ref{popped}. Our proof of Theorem \ref{maintheo} is based on the construction of a foliation 
$\F$ on the variety 
$\T:=V\times \A_\C^{\dim(F^{\frac{\mov}{2}})}$, where $F^{\frac{\mov}{2}}$ is the 
$\frac{\mov}{2}$-th piece of the Hodge filtration of the $\mov$-th cohomology 
bundle of $Y/V$. Surprisingly, 
this is not interpreted as the total space of the bundle 
$F^{\frac{\mov}{2}}$ which has been useful in the works \cite{cadeka, voisinHL}.
In \S\ref{foliationsection} we define the variety $\T$ and 
we construct the foliation $\F$.   
In $\T$ we define the Hodge locus with constant periods 
 and we show that it projects into the classical Hodge locus in $V$.
It turns out that the components of the Hodge locus
with constant periods  are leaves of $\F$. We discuss the leaves and singularities of
$\F$ in \S\ref{10oct2014} and \S\ref{10oct2014-1}, respectively. 
In \S\ref{barelaha} we prove our main results announced in the Introduction. 
This is based on a
precise description of the contribution of IVHS in the algebraic expression of
$\F$. The case of  Noether-Lefschetz locus 
is explained in \S\ref{NLlocus}. 
Finally, in \S\ref{19aug2014} we discuss
some problems and conjectures which might be useful for future works.

\subsection*{Notations} 
We explain some of our notations. We denote by $\hn^{ij},\ i+j=\mov$ the Hodge numbers of the $\mov$-th
cohomology of $Y_t$. The dimensions of the pieces of the Hodge filtration of $Y_t$ are denoted by
\begin{equation}
\label{saraghamgin}
 \hn^i:=\hn^{\mov,0}+\hn^{\mov-1,1}+\cdots+\hn^{i,\mov-i}.
\end{equation}
For a $\hn^0\times\hn^0$ matrix $M$ we denote by $M^{ij},\ i,j=0,1,2,\ldots,\mov$ the $\hn^{\mov-i,i}\times \hn^{\mov-j,j}$ submatrix of 
$M$  corresponding to the decomposition  $\hn^0:=\hn^{\mov,0}+\hn^{\mov-1,1}+\cdots+\hn^{0,\mov}$. 
We call $M^{ij},\ i,j=0,1,2\ldots,\mov$ the $(i,j)$-th Hodge block of $M$. In a similar way, for a $\hn^0\times 1$ matrix $M$
we write  $M=[M^i]$, where $M^{i},\ i=0,1,2,\ldots,\mov$ is the $\hn^{\mov-i,i}\times 1$ submatrix of $M$  
corresponding to the decomposition of $\hn^0$ into Hodge numbers. 
For any property "P" of matrices we say that the property "block P" or "Hodge block P" is valid if the property P is valid with respect to the Hodge blocks. For instance, we say that a $\hn^0\times\hn^0$ 
matrix $M$ is block upper triangular if 
$M^{ij}=0,\ \ i>j$. 
We always write a basis of a free $R$-module of rank $\hn^0$ or a vector space  as a $\hn^0\times 1$ matrix. Note that in \eqref{contradiction90}
we have already used the notation $a=\hn^{\frac{\mov}{2}+1,\frac{\mov}{2}-1}$ and $b=\hn^{\frac{\mov}{2},\frac{\mov}{2}}$. 

\subsection*{Acknowledgements}
The author would like to acknowledge the generous support of the
Brazilian science without border program for spending  
a sabbatical year at Harvard University. 
My sincere thanks go to S. T. Yau for the invitation and constant support.
I would like to thank D. van Straten for useful conversations on the first draft of the present paper. 
Thanks also go to J. Harris and C. Voisin  whose comments improved the text.
C. Voisin informed  the author that she was planning to 
approach Conjecture \ref{17july2014}  by studying second order behavior of the 
Noether-Lefschetz locus, but this appeared to be extremely technical and hard to exlpoit.
This idea has been formulated  in \cite{Maclean05}, Theorem 7. The advantage 
of our approach is its immediate consequences which are Theorem \ref{maintheo2} and Theorem \ref{maintheo3}.

\section{Modular foliations}
\label{foliationsection}
The theory of modular foliations in the sense that we use it here, was 
introduced in \cite{ho06-1}. In this section we give the necessary definitions
in order to handle a particular modular foliation constructed from a projective family. 

\subsection{Adding new parameters}
\label{08a2014}
Around any point of $V$ we can find  global
sections $\omega$ of the $\mov$-th relative de Rham cohomology sheaf of $Y/V$ such that $\omega$ at each fiber 
$H_\dR^*(Y_t), \ t\in V$ form a basis compatible with the Hodge filtration. If it is necessary we may
replace $V$ with a Zariski open subset of $V$. 
We take variables $x_1,x_2,\ldots, x_{\hn^{\frac{\mov}{2}}}$ and put them in a $\hn^0\times 1$ matrix $x$ as below.
The first $\frac{\mov}{2}$ Hodge blocks are zero and $x_i$'s are listed in the next blocks:
\begin{equation}
\label{desesperado}
x=\begin{pmatrix}
      0\\
      \vdots\\
      0\\
      x^{\frac{\mov}{2}}\\
      \vdots\\
      x^{\mov}
     \end{pmatrix}.
\end{equation}
We take $\cpe$ a non-zero evaluation of the matrix $x$ by some constants and call it a period vector. For instance, take $\cpe$
a vector with zero entries except for the entry corresponding to $x_1$, which is one.  
Let $\Smat$ be any  Hodge block lower triangular  $\hn^0\times \hn^0$ 
matrix depending on $x$ such that 
\begin{equation}
\label{08august2014}
\Smat\cdot \cpe=x
\end{equation}
and define 
$$
O:=\spec\left(
\C\left [x_1,x_2,\ldots, x_{\hn^{\frac{\mov}{2}}}, \frac{1}{\det(\Smat)}\right ]
\right). 
$$
We can take the matrix $\Smat$ the one  obtained from the identity matrix by replacing the $\hn^{\frac{\mov}{2}+1}+1$ column with $x$ 
in order to get the equality $\Smat\cpe=x$. In this way $\Smat^{-1}$ is obtained from $\Smat$ by 
replacing $x_1$ with $x_1^{-1}$ and $x_i,\ i\geq 2$ with
$-x_ix_1^{-1}$.
We consider the family $\X\to\T$, where $\X:=Y\times O,\ 
\T:= V \times O$. It is obtained from 
$Y\to V$ and the identity map $O\to O$. We also define $\alpha$ by 
\begin{equation}
\label{august2014}
\alpha:=\Smat^{-1}\cdot\omega.
\end{equation}
Let $\nabla: H^\mov_\dR(Y/V)\to \Omega_V\otimes_{\O_V}
H^\mov_\dR(Y/V)$ be the algebraic Gauss-Manin connection (see \cite{kaod68}). We can write $\nabla$
in the basis $\omega$ and define the $\hn^0\times\hn^0$ matrix $\gma$ by the equality: 
$$
\nabla\omega=\gma\otimes \omega.
$$
The entries of $\gma$ are differential 1-forms in $V$. In a similar way we can compute the Gauss-Manin connection  of $\X/\T$ in the basis 
$\alpha$:
$$
\nabla\alpha=\gm\otimes \alpha,
$$
where
\begin{equation}
 \label{omidtaiwan}
\gm=-\Smat^{-1} d\Smat+\Smat^{-1}\cdot \gma\cdot \Smat. 
 \end{equation}
This follows from the construction of the global sections $\alpha$ in \eqref{august2014} and the Leibniz rule.
We call $\gma$ (resp. $\gm$) the Gauss-Manin connection matrix of the pair $(Y/V, \omega)$ (resp. $(\X/\T,\alpha)$). 
 From the integrability of the Gauss-Manin connection it follows that
\begin{equation}
\label{nancy2014}
d\gm+\gm\wedge \gm=0.
\end{equation}

\subsection{Modular foliations}
\label{17aug2014}
Let $\T$ be an algebraic variety and $\gm$ be a $\hn^0\times\hn^0$ matrix whose entries are
differential 1-forms in $\T$ and it satisfies \eqref{nancy2014}. 
 For any $\hn^0\times 1$ matrix $\cpe$, the entries of $\gm\cpe$ induce  a holomorphic foliation $\F$ in $\T$. The integrability of the
distribution given by the kernel of the entries of $\gm \cpe$ follows from 
(\ref{nancy2014}): 
$$
 d(\gm\cdot \cpe)=-d\gm \cdot \cpe=\gm\wedge (\gm\cdot \cpe).
$$
Let $\F$ be a foliation given by a finite collection of differential $1$-forms $\alpha_i,\ \ i=1,2,\ldots,a$ in $\T$.
Its codimension $c$ is the dimension of the vector space generated by $\alpha_i$'s over the functions field of $\T$. 
Its singular set is defined to be
$$
\sing(\F):=\left \{t\in\T \mid 
 \alpha_{i_1}\wedge\alpha_{i_2}\wedge\cdots\wedge \alpha_{i_c}=0,\ \ \forall i_1,i_2,\ldots, i_c=1,2,\ldots,a 
\right  \}.
$$ 
The singular set $\sing(\F)$ is a proper algebraic subset of $\T$. An analytic irreducible 
(not necessarily closed) subset $L$ of $\T$ is tangent to $\F$ if it is 
tangent to the kernel of $\alpha_i$'s. It is called a (local) leaf of $\F$ if it is tangent to $\F$ 
and it is not a proper analytic subset of some $\tilde L$ tangent to $\F$.   
All the leaves of the holomorphic foliation $\F$ in $\T\backslash \sing(\F)$
have the same codimension $c$ and we call them general leaves. We call the others special leaves.
In the literature by a leaf one mainly means a general leaf.

Now, consider the Gauss-Manin connection matrix $\gm$ constructed in \S\ref{08a2014}. 
Let 
$$
\delta_t\in H^\mov(\X_t,\Q)\otimes_\Q \C,\ t\in(\T,t_0)
$$
be a continuous family of cycles, that is, $\delta_t$ is a flat section 
of the Gauss-Manin connection: $\nabla \delta_t=0$. Here, $(\T,t_0)$ is a small neighborhood of $t_0$ in $\T$ in 
the usual topology.
Let us define
\begin{equation}
\label{16aug2014}
L_{\delta_t}:= \left \{ \ \ \ t\in (\T,t_0) \ \ \ |  \ \ \ \langle \alpha,\delta_t\rangle=\cpe\right \},  
\end{equation}
where
\begin{equation}
\label{17f1393}
\langle\cdot,\cdot\rangle: H_\dR^\mov(\X_t)\times H_\dR^{\mov}(\X_t)\to \C,\ \ 
(\beta_1,\beta_2)\mapsto
\frac{1}{(2\pi i)^{\nov}}\int_{\X_t}
\beta_1\cup\beta_2\cup \theta^{\nov-\mov}
\end{equation}
and $\theta\in H_\dR^{2}(\X_t)$ is the element obtained by polarization.
We have the holomorphic function 
$$
f: (\T,t_0)\to \C^{\hn^0}, \ \ 
f(t):=\langle \alpha,\delta_t\rangle-\cpe
$$
which satisfies
\begin{equation}
 \label{28sep2014}
 df=\langle \nabla\alpha,\delta_t\rangle=\gm\cdot \langle \alpha,\delta_t\rangle=\gm\cdot \cpe+\gm\cdot f.
 \end{equation}
 This implies that $\gm\cdot  \cpe$ restricted to $L_{\delta_t}$'s is identically zero. More precisely, 
 the local leaves of $\F$
are given by $L_{\delta_t}$'s. 
Recall the constant period vector $\cpe$ defined in \S\ref{08a2014}. 
\begin{defi}\rm
The Hodge locus  with constant periods $\cpe$ is defined to be the union of all 
$L_{\delta_t}$ in \eqref{16aug2014}  
with $\delta_t\in H^\mov(\X_\t,\Q)$. 
\end{defi}
By definition any component 
of  the Hodge locus with constant periods  is either inside $\sing(\F)$  or 
it is a general leaf of $\F$. From the zero blocks of $\cpe$,
it follows that $\delta_t\in H^{\frac{\mov}{2},\frac{\mov}{2}}$ and so $\delta_t$ is a Hodge class.

\subsection{The algebraic description of modular foliations}
\label{hanoozhich}
We note that the foliation $\F$ in  $\T$ is given by
\begin{eqnarray}
\label{7aug2014}
0&=&\gma^{\frac{\mov}{2}-1,\frac{\mov}{2}} x^{\frac{\mov}{2}} \\
\label{7aug2014-2}
dx^\frac{\mov}{2} &=&  \gma^{\frac{\mov}{2},\frac{\mov}{2}} x^\frac{\mov}{2}+\gma^{\frac{\mov}{2},\frac{\mov}{2}+1} x^{\frac{\mov}{2}+1} ,
\\
\label{7aug2014-1}
dx^i &=& \sum_{j=\frac{\mov}{2}}^\mov \gma^{i,j} x^j,\ \ \ \ \ i=\frac{\mov}{2}+1,\ldots,\mov. 
\end{eqnarray}
 For this we use \eqref{omidtaiwan} and we conclude that $\F$ is given by
 $(-\Smat^{-1} d\Smat+\Smat^{-1}\cdot \gma\cdot \Smat)\cpe$. Since $\cpe$
is a constant vector and $\Smat$ is an invertible matrix and we have (\ref{08august2014}),  
we conclude that $\F$ is given by the entries of $dx-\gma x=0$. Opening this equality and
using the zero blocks of $x$ in (\ref{08august2014}) we
get \eqref{7aug2014}, \eqref{7aug2014-2} and \eqref{7aug2014-1}. 
Note that by Griffiths transversality $\gma^{i,j}=0$ for $j-i\geq 2$. Let 
\begin{equation}
\label{16oct2014}
\alpha:= \gma^{\frac{\mov}{2}-1,\frac{\mov}{2}}\cdot x^{\frac{\mov}{2}}.
\end{equation}
We will not use more $\alpha$ defined in \S\ref{08a2014}. 
Note that  $x^{\frac{\mov}{2}}$ is a $b\times 1$ matrix with unknown entries 
$x_i,\ \ i=1,2,\ldots,b$.
We consider the entries of $\alpha$ as differential forms in 
$V\times \A_\C^b$.
Let $c$  be the dimension of 
the vector space generated by the entries of $\alpha$ and over the functions field  of 
$V\times \A_\C^b$.
We define the algebraic set $\sinone_y,\ \ y=0,1,\ldots,c$ 
to be the Zariski closure of
\begin{equation}
\label{mano-zagier}
\left\{(t,x)\in V\times \A^b_\C \mid x\not=0,\ \ 
\alpha_{i_1}\wedge\alpha_{i_2}\wedge\cdots\wedge \alpha_{i_{c-y+1}}=0,\ \ \forall i_1,i_2,\ldots, i_{c-y+1}=1,2,\ldots,a 
\right  \}.
\end{equation}
We have inclusions of algebraic varieties
$$
\emptyset=\sinone_{c+1}\subset \sinone_{c} \subset \cdots \subset\sinone_1\subset \sinone_0=V\times \A^b_\C .
$$
The 
set $\sinone_y$ in \eqref{mano-zagier}  does not depend on 
the variables in $x^{i},\ i=\frac{\mov}{2}+1,\ldots,\mov$ and so we define
\begin{equation}
 \check\sinone_y:= \sinone_y\times  \A^{\hn^{\frac{\mov}{2}+1}}_\C 
\end{equation}
The affine variety $\T$ is a Zariski open subset of  $V\times \A^{\hn^{\frac{\mov}{2}}}_\C$  given by $x_1\not=0$.
From now on we redefine $\T$ to be $V\times \A^{\hn^{\frac{\mov}{2}}}_\C$. We have 
the foliation $\F$ in $\T$ given by the differential forms  \eqref{7aug2014},\eqref{7aug2014-1} and \eqref{7aug2014-2}.


\subsection{Singularities of modular foliations}
\label{10oct2014}

\begin{prop}
\label{trabalhatrabalha}
The set of singularities of the foliation $\F$ is given by $\check\sinone_1\cup \sintwo$, where 
$\sintwo\subset \T$ is given by $x^{\frac{\mov}{2}}=0$.   
\end{prop}
\begin{proof}
This follows from the explicit form \eqref{7aug2014}, \eqref{7aug2014-2} and \eqref{7aug2014-1} and the fact that
all the entries of $x^i$'s in \eqref{7aug2014-2} and \eqref{7aug2014-1} are independent variables.  
\end{proof}
We would like to understand the geometric meaning of the singular set $\sintwo$. 
Recall the definition of $L_{\delta_t}$'s in \eqref{16aug2014}. 
\begin{prop}
\label{manosad}
 There is no component of the Hodge locus with constant periods 
 inside $\sintwo$. 
\end{prop}
\begin{proof}
  Using the equalities
\eqref{august2014} and \eqref{08august2014},   
we know that a Hodge locus with constant periods
is given by $L_{\delta_t}: \langle \omega,\delta_t \rangle =x,\ \delta_t\in H^\mov(\X_t,\Q)$. 
The first $\frac{\mov}{2}$ Hodge 
blocks of $x$
are already zero and if $L_{\delta_t}\subset \sintwo$ then the 
next $(\frac{\mov}{2}+1)$-th  Hodge 
block is also zero.
If $\delta_t\in H^\mov(\X_t,\R)$ then we have 
$$
 \langle \omega^i,\delta_t \rangle  = \langle \overline{\omega^{i} },\delta_t \rangle=0,\ \ i=0,1,\ldots,\frac{\mov}{2},
$$
where $\overline{\omega}$ is the complex conjugation of the 
differential forms in the entries of  $\omega$.
Since the entries of $\overline{\omega^{i}}$ and $\omega^{i},  i=0,1,\ldots,\frac{\mov}{2}$ generate
the $\mov$-th complex de Rham cohomology of each fiber $\X_t,\ \ t\in \T$, 
we conclude that $\delta_t=0$. Therefore, 
for $L_{\delta_t}\subset \sintwo$, the cycle $\delta_t$ has not 
real coefficients. Note that a Hodge locus with constant periods
is defined using cohomology classes with rational coefficients.   
\end{proof}

\subsection{Leaves of modular foliations}
\label{10oct2014-1}
Let $\tilde r:=\dim V$. For simplicity, the reader can take an $r$-dimensional subvariety of $V$ such that the Kodaira-Spencer map
over its points is an isomorphism and so, follow the arguments with $\tilde r=r$. 
\begin{prop}
\label{28s2014}
 Any component of the analytic set $L_{\delta_t}$ which intersects $\T-\check\sinone_{y+1}$ has dimension $\leq \tilde r-c+y$. 
\end{prop}
\begin{proof}
Let $t$ be a point in $\T-\check\sinone_{y+1}$. We have $t\in \check\sinone_k\backslash \check \sinone_{k+1}$
for some $k$ in the set $\{0,1,\ldots,y\}$.
By definition of $\check\sinone_y$'s, 
the dimension of the $\C$-vector space $A$ spanned by the  differential forms  \eqref{7aug2014} ,\eqref{7aug2014-2}, \eqref{7aug2014-1} is exactly 
$\hn^{\frac{\mov}{2}}+c-k$, and so the kernel of such differential forms is of dimension $\tilde r-c+k$. Note that $\dim(\T)= \hn^{\frac{\mov}{2}}+\tilde r$.
\end{proof}

Recall the construction of $\T$ in \S\ref{08a2014} and 
let $\T\to V$ be the projection on $V$.  
 \begin{prop}
\label{unknownguy}
 Any component $H$ of the Hodge locus in $V$ is the projection under the map $\T\to V$ of a 
component $L$ of the Hodge locus with constant
 periods  in $\T$. Moreover, $\dim(H)=\dim(L)$.
 \end{prop}
 \begin{proof}
Let $\delta_{t_0}\in H^\mov(\X_{t_0}, \Q)\cap H^{\frac{\mov}{2},\frac{\mov}{2}}$ be a Hodge class. 
This is equivalent to say that
it is in the set 
$$
H:= \{ t\in (\T,t_0) \ \ \mid  \ \ \langle \omega^{i},\delta_t\rangle=0,\ \ \ i=0,1,\cdots, \frac{\mov}{2}-1 \}
$$ 
which is the Hodge locus passing through 
$t_0$.  The Hodge locus $L$ in $\T$ with the constant periods
 $\cpe$ is given by the set of $(t,x)\in \T$ such that
 $$
\langle \alpha,\delta_t\rangle=\Smat^{-1} \langle \omega,\delta_t\rangle=\cpe, 
 $$
 or equivalently 
\begin{equation}
\label{16a2014}
\langle \omega,\delta_t\rangle =x,
 \end{equation}
 where we have used \eqref{08august2014}. 
The first $\frac{\mov}{2}$ Hodge blocks of this equality are just the equalities in the definition
of $H$ and  hence $t$ must lie in the Hodge locus $H$. 
For others, the entries of the left hand side of \eqref{16a2014} are independent  variables and the 
entries of the right hand side are holomorphic functions in $H$ 
(from Deligne-Cattani-Kaplan theorem in \cite{cadeka} it follows that they are actually algebraic functions in $H$).
This implies that these equalities do not produce further constrains on $t$ and so the proposition is proved. 
 \end{proof}
 Note that if $H$ is a component of the Hodge locus in $V$  then Proposition \ref{28s2014} and Proposition \ref{unknownguy} imply 
 that the codimension of $H$ in $V$ is less 
than or equal to $a=\hn^{\frac{\mov}{2}-1,\frac{\mov}{2}+1}$ which is Proposition 5.14 in Voisin's book \cite{vo03}.

\section{Proofs}
\label{barelaha}
So far we have used the full Gauss-Manin connection 
in order to construct and study the modular foliation $\F$.
In this section we  remind which part of the Gauss-Manin
connection is IVHS. This will give us the proof of our main theorems.
Throughout the present section we will redefine $x$ to be the middle Hodge block 
$x^{\frac{\mov}{2}}$  of the matrix $x$ defined in \eqref{desesperado}.

\subsection{Infinitesimal variation of Hodge structures}
\label{16aug2014-2}
For definitions and details of  the concepts used below see \cite{CGGH1983} and \cite{voisinHL}.
Let 
$$
H^\mov_\dR(Y/V):=\cup_{t\in V}H_\dR^\mov(Y_t)
$$
be the algebraic 
de  Rham cohomology bundle of $Y/V$ and let $F^{k},\ \ k=0,2,\ldots,\mov+1$ be the subbundles of 
$H^\mov_\dR(Y/V)$  corresponding to Hodge filtration in its fibers. By Griffiths transversality theorem, 
the Gauss-Manin connection of $Y/V$ 
induces maps 
\begin{equation}
 \label{15aug2014}
\nabla_{k}: H^{k,\mov-k}\to \Omega_V^1\otimes_{\O_V} H^{k-1,\mov-k+1},\ \  \ k=1,2,\ldots, \mov,
\end{equation}
where $H^{k,\mov-k}:=F^k/F^{k+1}$. 
One usually use the canonical identifications
\begin{equation}
\label{asabkhord}
F^{k}/F^{k+1}\cong H^{\mov-k}(Y_t,\Omega_{Y_t}^k)
\end{equation}
compose the Gauss-Manin connection with vector fields in $V$ and  arrives at
$$
(T_V)_t\to {\rm Hom} (  H^{\mov-k}(Y_t,\Omega_{Y_t}^k),\   H^{\mov-k+1}(Y_t,\Omega_{Y_t}^{k-1}))
$$
Further, one may use a theorem of Griffiths which says that the above maps are the composition of the Kodaira-Spencer map \eqref{arezoyematoyito}  
and 
\begin{equation}
\label{IVHS}
\delta_{\mov,k}=\delta_k: H^1(Y_t,T_{Y_t})_\theta\to {\rm Hom} \left (  H^{\mov-k}(Y_t,\Omega_{Y_t}^k),\   H^{\mov-k+1}(Y_t,\Omega_{Y_t}^{k-1})\right )
\end{equation}
which is obtained by contraction of differential forms along vector fields. 
Hopefully, $\delta_k$
will not be confused with the topological cycle $\delta_t$ used in previous sections.  
Here $\theta$ is the element in $H^1(Y_t, \Omega_V^1)$ induced
by the polarization of $Y_t$ and 
$$
H_1(Y_t,T_{Y_t})_\theta:=\{v\in H_1(Y_t,T_{Y_t}) \mid \delta_{2,1}(v)(\theta)=0\}. 
$$
The data (\ref{IVHS}) is known as the infinitesimal variation 
of Hodge structures at $t$ (IVHS), see \cite{CGGH1983}. From this we get
\begin{equation}
\label{IVHS-dual}
\delta_k^*: H^{\mov-k+1}(Y_t,\Omega_{Y_t}^{k-1})^*\to {\rm Hom} \left ( H^1(Y_t,T_{Y_t})_\theta,  H^{\mov-k}(Y_t,\Omega_{Y_t}^k)^*  \right )
\end{equation}
Let $\omega$ be the basis of  $H^\mov_\dR(Y/V)$ chosen in \S\ref{08a2014}. 
This induces a basis for both $H^{\mov-k+1}(Y_t,\Omega_{Y_t}^{k-1})$ and 
$H^{\mov-k}(Y_t,\Omega_{Y_t}^k)$ which we denote them by $\omega_*$ 
and $\varpi_*$, 
respectively. 

Around a smooth point of $V$ we choose coordinate system  $(t_1,t_2,\ldots,t_{\tilde r})$ and we denote by the same notation 
the image of the vector field 
$\frac{\partial}{\partial t_i},\ \ i=1,2,\ldots, \tilde r$ under the Kodaira-Spencer map \eqref{arezoyematoyito}.
Let $\gma$ be the Gauss-Manin connection matrix of $Y/V$
written in the basis $\omega$ and used in \S\ref{08a2014}. 
\begin{prop}
\label{13oct2014-1}
 We have
 \begin{equation}
 \label{mikhaminoyejayibebaram}
 \gma^{k,k-1}=\sum_{j=1}^{\tilde r} \gma_j^{k,k-1} dt_j
 \end{equation}
 where $\gma_j^{k,k-1}$ is the $a\times b$ matrix of $\delta_k( \frac{\partial}{\partial t_j} )$
 written in the bases $\omega_*$ and $\varpi_*$. 
\end{prop}
\begin{proof}
 This follows from the identifications \eqref{asabkhord}.
\end{proof}
For the proposition below we set $k=\frac{\mov}{2}+1$. 
\begin{prop}
\label{13oct2014-2}
For $y=0,1,\ldots,c$ the determinantal variety of homomorphisms of rank $\leq c-y$  
of \eqref{IVHS-dual} 
is given by $\sinone_y$ defined in \eqref{mano-zagier}. 
\end{prop}
\begin{proof}
Let $\gma=\gma^{k,k-1}$. 
 We have $\gma\cdot x= \sum_{j=1}^{\tilde r} (\gma_jx)dt_j$ and
 \begin{equation}
 \label{contr-conti}
 [\gma_1x,\gma_2x,\cdots, \gma_{\tilde r}x]= 
 \sum_{j=1}^b x_j[\gma_1^j,\gma_2^j,\ldots, \gma_{\tilde r}^j], 
 \end{equation}
 where $\gma_{j}^h$ is the $h$-th column of $\gma_j$.
 By definition $\sinone_y$ is the determinantal
 variety of matrices  
of rank $\leq c-y$ constructed from the left hand side of \eqref{contr-conti}.
 The determinantal variety of 
the right hand side of \eqref{contr-conti} is the determinantal variety of \eqref{IVHS-dual}.  
\end{proof}

\subsection{Proof of Theorem \ref{maintheo} using modular foliations}
Let $y+1:=c-s$. 
By our assumption the fiber of $\pi_{y+1}: \check \sinone_{y+1}\to V$ over $t$
is empty. The variety $\sinone_{y+1}$ is given by homogeneous polynomials in $x$ and with
coefficients in $\O_V$. 
Since projective varieties are complete, the image of $\pi_{y+1}$ is a 
closed proper subset of $V$ which does not contain $t$, see for instance Milne's lecture notes \cite{milneAG}. 
Let  $U\subset V$ be the complement  of the image of $\pi_{y+1}$ in $V$ and so $t\in U$.
By Proposition \ref{unknownguy} we know that a component of the Hodge locus $H$ in $U$ is the projection
of a component $L$ of the Hodge locus with constant periods  $L_{\delta_t}$ in $\T-\check\sinone_{y+1}$ and $\dim H=\dim L_{\delta_t}$.
Using Proposition \ref{28s2014} we get $\dim (H)\leq \tilde r-c+y$. Therefore, the codimension of $H$ in $V$ is $\geq s+1$.


\subsection{IVHS for hypersurfaces}  
Let $V$ be the parameter space of smooth  hypersurfaces of degree $d$ in $\Pn {\mov+1}$.
For $t\in V$ let 
$f_t(X_0,X_1,\ldots,X_{\mov+1})$ be 
the corresponding homogeneous 
polynomial, $Y_t:=\Pn {}\{f_t=0\}$.
We have the identifications
\begin{eqnarray} \label{12oct2014-1}
H_1(Y_t,T_{Y_t})_\theta & \cong & (\C[X]/J)_d\\ \label{12oct2014-2}
H^{\mov-k,k}    &\cong  &  (\C[X]/J)_{(k+1)d-\mov-2},\ \ k=0,1,\ldots,\mov
\end{eqnarray}
 where $J:= \jacob(f_t)$ is the Jacobian ideal of $f_t$ (for $k=\frac{\mov}{2}$ one must use the primitive part of $H^{\frac{\mov}{2},
 \frac{\mov}{2}}$).
 Note that we have changed the role
 of $k$ in \S\ref{16aug2014-2} with $\mov-k$. For $g$ in the right hand
 side of \eqref{12oct2014-1} the corresponding deformation of $f_t$ is given by $f_t+\epsilon g$ and for $g$
 in the right hand side of \eqref{12oct2014-2} we have 
 \begin{equation}
 \label{omidam}
 \omega_g:={\rm Residue}\left( 
 \frac{g\cdot \sum_{i=0}^{\mov+1}(-1)^i x_i\ dx_0\wedge \cdots \wedge \widehat dx_i\wedge \cdots\wedge dx_{\mov+1}}{f_t^{k+1}} \right)\in H^{\mov-k,k},
 \end{equation}
where we have used the residue map $ H^{\mov+1}_\dR(\Pn {\mov+1}\backslash Y_t)\to H_\dR^{\mov}(Y_t)$, see for instance \cite{gr69}. 
After these identifications, 
the map $\delta_k$ in \eqref{IVHS} turns out to be obtained by multiplication 
of polynomials, that is, we get
\begin{equation}
\label{08sep2014}
(\C[X]/J)_d\times (\C[X]/J)_{(k+1)d-\mov-2}\to (\C[X]/J)_{(k+2)d-\mov-2},\ \  (F,G)\mapsto FG,
\end{equation}
see \cite{CGGH1983, Harris85}.  Due to Hodge classes we are mainly interested in 
$\mov$ even and $k=\frac{\mov}{2}-1$. Recall the set $I_N$ in \eqref{21oct2014} and
let 
$I$ be the union of all $I_N$'s. We are going to work in a neighborhood
of the Fermat point $0\in V$ and so we take
\begin{equation}
\label{khodmidanad}
f_t:= X_0^{d}+X_1^{d}+\cdots+X_{\mov+1}^d-d\cdot \sum_{j \in I_d}t_j X^j ,\ \ \ 
\end{equation}
and work with 
$V=\spec\left(\C[t, \frac{1}{\Delta}]\right)$, where $t=(t_j)_{j\in I_d}$ and  
$\Delta(t)=0$ is the locus of parameters $t\in V$ such that
the monomials
\begin{equation}
\label{fermatbasis}
X^j:=X_0^{j_0}X_1^{j_1}\cdots X_{\mov+1}^{j_{\mov+1}},\ \ \  0\leq j_e\leq d-2
\end{equation}
do not form a basis of $\C[X]/J$.  
The matrices $\gma_j^{k,k-1}$ in \eqref{mikhaminoyejayibebaram} have entries in  $\C[t, \frac{1}{\Delta}]$.
The Kodaira-Spencer map \eqref{arezoyematoyito} in this case is an isomorphism of fiber bundles in $V$. In general, it is always useful
to choose a coordinate system  $(t_1,t_2,\ldots,t_r)$ such that the image of the vector field 
$\frac{\partial}{\partial t_j},\ \ j=1,2,\ldots, r$ under $(T_V)_t\to H^1(Y_t,T_{Y_t})_\theta$ form a basis 
of $H^1(Y_t,T_{Y_t})_\theta$.

\subsection{IVHS for the Fermat variety}  
\label{geryekardam}
Let us consider the case $t=0$, that is, we are going to deal with IVHS of the Fermat variety.
In this case, the map (\ref{20sep2014-1}) can be computed easily. A monomial $X^i\in \C[X]/J$ is zero
if and only if for some $e$, $i_e\geq d-1$. Therefore, for $(X^j,X^i)\in
(\C[X]/J)_d\times (\C[X]/J)_{(k+1)d-\mov-2}$, the product $X^{i+j}$
is a member of the canonical basis of $(\C[X]/J)_{(k+2)d-\mov-2}$ if 
$i_e+j_e\leq d-2$ for all $e$ and it is zero otherwise. 
From now on we use $I_N$ as an index set. Take variables 
$$
x_i,\ \ i\in I_{(k+2)d-\mov-2}.
$$
Assume that for any other $i$ which is not in  $I_{(k+2)d-\mov-2}$, $x_i$ by 
definition is zero. 
The matrix $\gma_j^{k,k-1},\ j\in I_d$ is therefore a $a\times b$-matrix with entries $0$ everywhere except at
$(i,i+j)$ entries which is one. Therefore, 
We have
\begin{equation} 
 \gma_j^{k,k-1} x=[x_{i+j}],\ \ j\in I_d,\ \ i \in I_{(k+1)d-\mov-2}
\end{equation}
and we can compute the $a\times r$ matrix:  
\begin{equation}
\label{lovelymatrix}
M^{k,k-1}:=
[\gma_1^{k,k-1}x,\gma_2^{k,k-1}x,\ldots,\gma_r^{k,k-1}x]=[x_{i+j}],\ \ \ j \in I_d,\ \ i \in I_{ (k+1)d-\mov-2}
\end{equation}
where $i$ counts the rows and $j$ counts the columns.
The differential forms $\alpha_j$ defined in \eqref{16oct2014} with $k=\frac{\mov}{2}-1$ 
and evaluated 
at the Fermat point $t=0$ are given by
$$
\alpha_i=\sum_{j\in I_d} x_{i+j}dt_j,\ \ \ i\in I_{(k+1)d-\mov-2}.
$$

\subsection{Proof of Theorem \ref{maintheo2}}
Throughout the present section set $k:=\frac{\mov}{2}-1$ for the content of \S\ref{geryekardam}.
We will reuse the letter $k$ for an element in $I_{ (\frac{\mov}{2}+1)d-\mov-2}$. 
We get the following Olympiad problem:
\begin{prop}
\label{yademadaram}
 If 
 \begin{equation}
 \label{happyandsad}
 \rank([x_{i+j}])<\bn{\frac{\mov}{2}+d}{d}-(\frac{\mov}{2}+1)^2
 \end{equation}
 then all $x_i, i\in I_{(\frac{\mov}{2}+1)d-(\mov+2)}$ are zero. 
\end{prop}
\begin{proof}
 Take any additive ordering for  $\N_0^{\mov+2}$, that is, 
 $i<j$ if and only if $i+k<j+k$ for all $i,j,k\in \N_0^{\mov+2}$. For instance take the lexicographical ordering.
 We use decreasing  induction on $k\in \Z^{\mov+2},\ \ |k|=(\frac{\mov}{2}+1)d-(\mov+2)$.
 For very big $k$, we have $k\not \in I_{(\frac{\mov}{2}+1)d-(\mov+2)}$ and so
 we have automatically $x_k=0$. Let us assume that $x_{\tilde k}=0$ for all $\tilde k >k$. We collect all
 $$
 k=i_e+j_e,\  i_e \in I_{ \frac{\mov}{2}d-\mov-2},\  j_e\in I_d,\ \ \ e=1,2,\cdots, f
 $$ 
 and order them
 according to the decreasing order of $i_e$'s, that is, $i_{e_1}>i_{e_2}>\cdots$.
 We find a $f\times f$-submatrix of $[x_{i+j}]$  which is lower triangular and in its diagonal
 we have only $x_k$. Therefore, its determinant is $x^f_k$. Now our proposition  follows from an even more
 elementary problem.
\end{proof}
\begin{prop}
 For any $k\in I_{(\frac{\mov}{2}+1)d-(\mov+2)}$ we have
 \begin{equation}
 \label{madaram}
 \#\left \{(i,j)\in I_{\frac{\mov}{2}d-(\mov+2)}\times I_d \mid k=i+j \right \}\ \ \ \geq \ \ \ 
 \bn{\frac{\mov}{2}+d}{d}-(\frac{\mov}{2}+1)^2
 \end{equation}
\end{prop}
\begin{proof}
 Let $A_k$ be the set in the  left hand side of \eqref{madaram} for $k=(k_0,k_1,\ldots,k_{\mov+1})$.
 It is easy to see that the lower bound in \eqref{madaram} is
 obtained by elements $k$ such that $\frac{\mov}{2}+1$ number of $k_e$'s are zero and the rest (exactly the next half)
 is $d-2$.  Let us assume that $k$ is not of the mentioned format. For simplicity we can assume that
 $0<k_0<k_1<d-2$. We prove that 
 \begin{equation}
 \label{yekeshab}
 \#  A_{(k_0,k_1,\cdots)}\geq \#A_{(k_0-1,k_1+1,\cdots)}
 \end{equation}
 and so repeating the same argument for $(k_0-1,k_1+1,\cdots)$ we get an element with only $d-2$ and $0$ as its entries. 
 In order to prove \eqref{yekeshab}  we define a map
 $$
 A_{(k_0-1,k_1+1,\cdots)}\to A_{(k_0,k_1,\cdots)}
 $$ 
 and prove that it is injective. It sends the pair  \(\left((i_0,i_1,\ldots),(j_0,j_1,\ldots)\right )\) to the pair  
 $((i_0+1,i_1-1,\ldots)$, $(j_0,j_1,\ldots))$ if \(i_1\not=0\) and to \(\left((0,i_0,\ldots),(k_0,k_1-i_0,\ldots)\right)\) 
 if \(i_1=0\). 
 It is easy to see that this map is  injective and so \eqref{yekeshab} is valid. 
\end{proof}

In Proposition \ref{yademadaram} the number in the right
 hand side of \eqref{happyandsad} is the biggest one with such a  property. It
 is enough to find complex numbers $x_k$ such that $\rank ([x_{i+j}])$ is the number
 in the right hand side of \eqref{happyandsad}. Such numbers are the periods of the projective space
 $\Pn {\frac{\mov}{2}}$ inside the Fermat variety $V_0$ given by $x_0-\zeta x_1=x_2-\zeta x_3=\cdots=x_{\mov}-\zeta x_{\mov+1}=0$, where 
 $\zeta^d+1=0$. More precisely
 $$
 x_k:=\int_{\Pn {\frac{\mov}{2}}} \omega_{g_k}, 
 $$
 where $g_k:=X_0^{k_0}X_1^{k_1}\cdots X_{\mov+1}^{k_{\mov+1}}$ and $\omega_{g_k}$ is defined in \eqref{omidam}.


\subsection{Proof of Theorem \ref{maintheo3}}
 \label{tanhayetanha}
We compute the first approximation of IVHS
 around the Fermat point. 
 More precisely, we compute the image of the Zariski tangent space of $\sinone_{c-s}$ 
 at each point $(0,x)$ and under the derivation of the projection map $\pi: \sinone_{c-s}\to V$.
 We prove  that for any $s\leq s_{\rm max}^1$ and $(0,x)\in\pi^{-1}(0),\ x\not=0$, 
 the derivation of $\pi$, which maps the Zariski tangent space of $\sinone_{c-s}$ at $(0,x)$
 to the Zariski tangent space of $V$ at $0$, is not surjective. This implies that $\pi$ is not dominant. 
 Note that the fibers of $\pi$ are given by homogeneous polynomials and so one may take
 their projectivization. 

 Let us consider the IVHS \eqref{08sep2014} for the polynomial $f_t$ in \eqref{khodmidanad} with
 fixed parameters $t=(t_\alpha)_{\alpha\in I_d}$ and arbitrary $k$. 
 We need to emphasize that the matrix $\gma^{k,k-1}$ depends on $t$ and so
 we write $\gma^{k,k-1}(t):=\gma^{k,k-1}$. In this way, for fixed $j\in I_d$, 
 $\gma_j^{k,k-1}(t)$ is the $a\times b$ matrix of the IVHS \eqref{08sep2014}  
 written in the basis \eqref{fermatbasis}.
 We write the Taylor series of the matrix $\gma_j^{k,k-1}(t)$ in the variables 
 $t_\alpha,\ \alpha\in I_d$:
 \begin{equation}
  \gma_j^{k,k-1}(t)=\gma_j^{k,k-1}(0)+\sum_{\alpha\in I_d}    \gma_{j,\alpha}^{k,k-1}(0)\cdot    t_\alpha+\cdots
 \end{equation}
 where $\cdots$ means sum of  homogeneous polynomials of degree $\geq 2$ in $t_\alpha$'s 
 and with coefficients depending on $x$.
 Let $K$ be the ideal of $\C[t_\alpha, \ \alpha\in S_d]$ generated by 
 $t_\alpha t_\beta,\ \ \alpha,\beta \in S_d$. 
 Therefore, $\cdots$ in the above equality 
 means it belong to $K$. 
 As we mentioned in \S\ref{geryekardam}, $\gma_j^{k,k-1}(0)$ is a $a\times b$ matrix with $1$ in 
 its $(i,i+j)$ entries and $0$ for entries elsewhere.
 In order to compute $\gma_{j,\alpha}^{k,k-1}(0)$ we notice that
 \begin{equation}
 \label{sili}
X^iX^j = \left\{ 
\begin{array}{ll}
         X^{i+j} & \forall e\ \  i_e+j_e\leq d-2, \\
         \sum_{\alpha\in I_d} t_\alpha \cdot  \alpha_{\check e} \cdot \frac{X^{i+j+\alpha}}{X_{\check e}^d}  &  \forall e\ \  i_e+j_e\leq d-2, 
         \hbox{ except for exactly one } e=\check e,  \\
         0  & \hbox{ otherwise.}
         
         \end{array} \right.
\end{equation}
The last two equalities are written modulo both ideals $\jacob(f_t)$ and $K$. From this 
we derive the fact that the matrix 
$\gma_{j,\alpha}^{k,k-1}(0)$ in its $(i, i+_\alpha j)$ entry has $\alpha_{\check e}$ for those $(i,j)$
in the second equality in \eqref{sili}, and elsewhere entries are zero.
We define an $a\times r$ matrix
$\check M_\alpha^{k,k-1}$ in the following way. 
For $(i,j)$ in the second equality  of \eqref{sili}, 
the $(i,j)$ entry of $\check M_\alpha^{k.k-1}$ is
$\alpha_{\check e}\cdot x_{i+_\alpha j }$, and elsewhere entries  are zero.
In this way, $\gma_{j,\alpha}^{k,k-1}(0)\cdot x$ is the $j$-th column of $M^{k,k-1}_{\alpha}$.
Let  $N_{j,\alpha}^{k,k-1}$ be the $a\times r$ matrix obtained 
by replacing the $j$-th column of $M^{k,k-1}$ with the $j$-th column of $\check M_\alpha^{k,k-1}$.
 
From now on, we set $k=\frac{\mov}{2}-1$, we do not write
the $k,k-1$ upper index of our matrices and we have the same notation as in the Introduction.  
 The variety $\sinone_{c-s}$ is given by $(s+1)\times (s+1)$ minors of the $a\times r$ matrix 
$$
\left [\gma_1^{k,k-1}(t)\cdot x\ ,\ \gma_2^{k,k-1}(t)\cdot x\ ,\ \ldots\ ,\ \gma_r^{k,k-1}(t)\cdot x \right ]
$$
and so the Zariski tangent space of $\sinone_{c-s}$ 
under the derivation of $\pi$
maps to
$$
\left\{ (v_\alpha)_{\alpha\in S_d} \ \ \mid\ \  \sum_{\alpha\in S_d} 
\left( \sum_{j\in I_d} {\rm minor}_{s+1}(N_{j,\alpha}^{k,k-1})  \right )v_\alpha=0 \right\}.
$$
Now, we use the  definition of $s_{\rm max}^1$ and Theorem \ref{maintheo3} is proved.

\subsection{Noether-Lefschetz locus}
\label{NLlocus}
Max Noether's theorem asserts that every curve on a general
hypersurface $X$ of degree $d\geq 4$ in $\Pn 3$ is a complete
intersection with another surface. In other words,
the Picard group  
of $X$ is free of rank one. 
Noether's argument was just the
plausibility of the statement and the first rigorous proof of this
theorem was given by S. Lefschetz in \cite{lefschetz1924} by using a
monodromy argument. A new Hodge-theoretic proof is given by
Carlson, Green, Griffiths and Harris in \cite{CGGH1983, GH1985}.
Let $V\subset \Pn {N}$ be the parameter space of smooth hypersurface $X$ of degree $d\geq 4$ in 
$\Pn 3$.
Noether-Lefschetz locus $\NLL_d$ in $V$ is the locus of smooth surfaces with Picard group 
different from $\Z$. By Lefschetz $(1,1)$ theorem any two dimensional Hodge class in $X$ is algebraic 
and so $\NLL_d$ is a particular example of a Hodge locus.  
 It is a countable
union of proper algebraic subset of $V$ and a surface with parameters outside $\NLL_d$
is called a general surfaces. Let $H$ be an irreducible component of $\NLL_d$ and $\codim_V(H)$ 
be its codimension in $V$.
We have  
$$
d-3\leq \codim_V(H)\leq \bn{d-1}{3}
$$
where the upper bound is the $(2,0)$ Hodge number $\hn^{20}$ of $X$.  
Ciliberto, Harris and Miranda in \cite{CHM88} proved that for $d\geq 4,$ 
$\NLL_d$
contains infinitely many general components, that is those components $H$ such that $\codim_V(H)=\hn^{20}$, 
and  the union of these components is Zariski dense in $V$. 
Components of codimensions $<\hn^{20}$ are called special (or exceptional). 
Green and Voisin in a series of paper showed that $d-3$ is the minimum codimension for
the components of $\NLL_d$ and for $d\geq 5$ the only component of codimension $d-3$ is the 
family of surfaces
containing a line, see \cite{green1988, green1989, voisin1988}. Voisin in \cite{voisin89} showed 
that for $d\geq 5$ the second biggest component of 
$\NLL_d$ is of codimension $2d-7$, which consists of those surfaces containing a conic.  
This implies a conjecture of J. Harris is true in the case $d=5$: 
there should be only finitely many special components of  $\NLL_d$. 
More evidences for Harris' conjecture came from the work   
\cite{cox90} of Cox in the case of elliptic surfaces, Debarre and Laszlo's work \cite{DL1990} 
for abelian varieties  and Voisin's work \cite{voisin90} for hyepersurfaces 
of degree $d=6,7$. For $d=4s$ arbitrarily large, Voisin in \cite{voisin1991} 
constructed an infinity of special components for $\NLL_d$ and hence gave counterexamples to 
Harris' conjecture. All such components are 
contained in a proper algebraic subvariety of the set of algebraic surfaces of degree $d$ 
and so she formulated Conjecture \ref{17july2014}. In this case the numbers \eqref{contradiction90} are given by
\begin{eqnarray}
 \label{vasacomer}
c=a&=&\bn{d-1}{3}\\ \nonumber
b &=& (d-1)^3-(d-1)^2+(d-1)-2\bn{d-1}{3}\\ \nonumber
r&=&\bn{d+3}{3}-16.
\end{eqnarray}
General components are of codimension $a$ and so using Proposition \ref{unknownguy} 
the modular foliation $\F$ in $\T:=V\times \A_\C^{a+b}$ is of the maximal codimension $a$. 
The differential forms \eqref{7aug2014}
are linearly independent over the function field of $\T$ and so $c=a$ and  $\sinone_1$ is the Zariski 
closure of 
\begin{equation}
\label{10s2014}
\left\{(t,x)\in V\times \A^ {b}_\C \mid x\not=0,\ \ \ \alpha_{1}\wedge\alpha_{2}\wedge\cdots\wedge \alpha_{a}=0\ \right  \}.
\end{equation}

\subsection{K3 surfaces}
Let us consider the case  of hypersurfaces  $X$ of degree $4$ in $\Pn 3$, that is,  $\mov=d-2=2$.
In this case $X$  is called a K3 surface. We have canonical identifications
$H^{2,0}\cong\C,\ H_1(Y_t,T_{Y_t})_\theta\cong H^{11}_{\rm prim}\cong (\C[X]/J)_d $.  We choose a basis of 
$(\C[X]/J)_d$ and hence we obtain a basis of both $H_1(Y_t,T_{Y_t})_\theta$ and 
$H^{11}_{\rm prim}$. Using  (\ref{08sep2014}) we get
$$
\gma^{0,1}=[dt_1,dt_2,\ldots,dt_b],\ \ \alpha_1:=x_1\cdot dt_1+x_2\cdot dt_2+\cdots+x_b\cdot dt_b.
$$
We conclude that  the variety $\sinone_1$  is empty.
The singular set of the modular foliation $\F$ in $\T$ is given by $\sintwo$. An expert
in holomorphic foliations might be interested to know whether the components of the
Noether-Lefschetz locus with constant periods, are the only algebraic leaves of $\F$. 
A similar discussion is also valid for the case of four dimensional cubic hypersurfaces, that is, $d=3, \mov=4$.
In this case $a=1, b=r=20$. Note that in this case the Hodge conjecture is well-known, see \cite{zu77}.

\subsection{Conjecture \ref{paperyounes} and the proof of Corollary \ref{estamuydirty}}
\label{conjsection}
We know that the codimension of the determinantal variety $D_{s,t}$ of homomorphisms of rank $\leq s$ 
in the right hand side of  \eqref{IVHS-2014} is 
$(a-s)(r-s)$. Therefore, we may hope that for generic $t$ 
the map \eqref{IVHS-2014} is transversal to $D_{s,t}$ (it meets  $D_{s,t}$ properly in the terminology of 
\cite{eisenbud1988}). 
If this happens then 
the codimension of the pull-back $W_{c-s,t}$ of $D_{s,t}$ in the left hand side
of \eqref{IVHS-2014}
is the same as the codimension of $D_{t,s}$ in the right hand side of \eqref{IVHS-2014}. 
Here,  the codimension of the  empty set  is defined to be any number bigger than the dimension of the ambient space. 
Therefore, 
the projectivization of the map \eqref{IVHS-2014} does not intersect $D_{s,t}$ if 
\begin{equation}
\label{29sept2014}
b\leq  (a-s)(r-s)=s^2-(r+a)s+ar.
\end{equation}
The biggest $s$ which satisfies this property is $\check s_{\rm max}$ defined in \eqref{30sept2014}.
Corollary \ref{paperyounes} follows from the computation of $\check s_{\rm max}$ from the data in \eqref{vasacomer}. 


\subsection{Second proof of Theorem \ref{maintheo} }
\label{popped}
Since determinantal varieties are homogeneous and projective varieties are complete, 
there is a Zariski open neighborhood $U$ of $t\in V$ such that the hypothesis of Theorem \ref{maintheo}
is valid for all points in $U$ and so  it is enough to prove that a component $H$ of the Hodge locus
passing through $t$ has codimension $\geq s+1$.

First, note that the algebraic cup product in de Rham cohomology \eqref{17f1393} after canonical identifications \eqref{asabkhord} 
gives  us isomorphisms
\begin{equation}
\label{delshekaste}
 H^{\mov-k}(Y_t,\Omega_{Y_t}^{k})^*\cong   H^{k}(Y_t,\Omega_{Y_t}^{\mov-k}),\ \ \ k=0,1,\ldots,\mov
\end{equation}
and under these isomorphisms, the map \eqref{IVHS-2014} constructed from the $(\mov-k)$-th
IVHS is identified with the $k$-th IVHS: 
\begin{equation}
\label{badankoofte}
H^{k-1 }(Y_t,\Omega_{Y_t}^{\mov-k+1})\to {\rm Hom} \left (H^1(Y_t,T_{Y_t})_\theta,  
H^{k}(Y_t,\Omega_{Y_t}^{\mov-k})  \right ).
\end{equation}
Composing the right hand side of \eqref{badankoofte} with the Kodaira-Spencer map \eqref{arezoyematoyito}
one arrives at Voisin's $^t\bar\nabla$ map:
\begin{equation}
\label{badankoofte-1}
^t\bar\nabla: H^{k-1 }(Y_t,\Omega_{Y_t}^{\mov-k+1})\to {\rm Hom} \left ((T_V)_t,  
H^{k}(Y_t,\Omega_{Y_t}^{\mov-k})  \right ).
\end{equation}
Now consider the case $k=\frac{\mov}{2}+1$.
Let $\delta_t\in H_\dR^\mov(Y_t)$
be the Hodge class whose locus  is the component $H$. It induces an element $\delta^{\frac{\mov}{2},\frac{\mov}{2}}_t\in  
H^{k-1 }(Y_t,\Omega_{Y_t}^{\mov-k+1}) $ and Voisin in \cite{vo03} 5.3.3 has shown that $\ker( ^t\bar\nabla \delta^{\frac{\mov}{2},\frac{\mov}{2}}_t )$
is the Zariski tangent space of $H$ at $t$. Therefore
$$
\codim_V H\geq \dim V-\dim \left( \ker( ^t\bar\nabla \delta^{\frac{\mov}{2},\frac{\mov}{2}}_t )\right)\geq 
\rank\left ( ^t\bar \nabla \delta^{\frac{\mov}{2},\frac{\mov}{2}}_t\right ) \geq s+1.
$$

\section{Final remarks}
\label{19aug2014}

We can use \eqref{august2014} and we can interpret  
the variety $\T$ defined in \S\ref{08a2014} in the following way: 
Let $\tilde\T$ be the total space of the 
$F^{\frac{\mov}{2}}$ bundle mines the total space of $F^{\frac{\mov}{2}+1}$ bundle. 
In $\tilde\T$ we have a canonical line bundle $L$ 
obtained by the choice of an element in $F^{\frac{\mov}{2}}$. Let us denote by $F^{i}$
the pull-back of the bundles with the same name by the projection $\tilde\T\to V$. 
For $i\leq \frac{\mov}{2}$ we have a canonical embedding 
$L\subset F^i$ and we define
$\tilde F^i=F^i/L$. For other cases we define $\tilde F^i=F^i$.
 The variety $\T$ is the total space of choices of bases for 
 $F^i$'s and $L$  with a fixed variation in all $\tilde F^i$'s.  
 These kind of total spaces in a refined format of  moduli spaces
were extensively used by the author to give geometric interpretation of quasi-modular 
forms and to construct analytic objects which transcend the classical automorphic forms, 
see \cite{ho18} and the references therein. The notion of modular foliations based on the
historical examples of Darboux, Halphen and Ramanujan has been introduced by the author in 
\cite{ho06-1}.

A systematic solution to the  Harris-Voisin conjecture and its generalizations involves  the  
study of the Hodge block  $\gma^{\frac{\mov}{2}, \frac{\mov}{2}}$
of the Gauss-Manin connection used in the algebraic expression of the foliation $\F$ given in  
\eqref{7aug2014}, \eqref{7aug2014-2}, \eqref{7aug2014-1}. Note that this is not covered in the IVHS. 
This together with two other matrices computed from  IVHS, namely $\gma^{\frac{\mov}{2}-1, \frac{\mov}{2}}, \gma^{\frac{\mov}{2}, \frac{\mov}{2}+1}$,
govern the codimensions of the leaves of $\F$. 

We know that the number $c$ is less than or equal the minimum of $r$ and $a$. In the case of hypersurfaces of dimension two
we saw that $c=a<r$. 
For hypersurfaces of dimension $\geq 4$ and degree $d\geq 2\frac{\mov+2}{\mov-2}$ we have $r\leq a$ and one may conjecture that 
$c=r$. For the $r$-dimensional parameter space in \eqref{khodmidanad}, this implies that
the foliation $\F$ is  zero dimensional 
and so a generic hypersurface has isolated Hodge classes with constant periods.
However, this does not imply the following stronger statement: 
\begin{conj}
There is a Zariski open subset $U$ of the parameter space $V$ of smooth hypersurfaces of degree $d\geq 2\frac{\mov+2}{\mov-2}$ in $\Pn {\mov+1}, \ \mov\geq 4$
such that  Hodge classes of $Y_t,\ t\in U$ are isolated, that is, all the components of the Hodge locus in $U$  are the orbits
of $\rm PGL(\mov+2,\C)$ acting on $V$. 
\end{conj}
Note that IVHS, and hence Theorem
\ref{maintheo}, is not enough
for verifying this conjecture as we need to study the Hodge block 
$\gma^{\frac{\mov}{2}, \frac{\mov}{2}}$ of the Gauss-Manin connection. 
Even if Conjecture
\ref{paperyounes} is true, the number $\check s_{\rm max}+1$ defined in \eqref{30sept2014} cannot 
be $\geq r$.

\href{http://w3.impa.br/~hossein/WikiHossein/files/Singular%20Codes/2014-10-GMCD-NL.txt}
{
We were not able to verify the hypothesis of Theorem \ref{maintheo} for some examples of hypersurfaces 
and with  the help of a computer 
code which uses Gr\"obner basis for ideals.} Even in the case of the Fermat variety, where
 we were  able to compute IVHS and prove Proposition \ref{yademadaram} all without any computer 
 assistance,
 writing a simple minded code to prove Proposition \ref{yademadaram} fails.
Such computational difficulties deserve
to be treated in a separate work. There are few examples which one might be able to  treat them 
by hand. One of them is the famous  Dwork family of Calabi-Yau varieties:
\begin{equation}
 X_0^{\mov+2}+X_1^{\mov+2}+\cdots+X_{\mov+1}^{\mov+2}-t\cdot X_0X_1\cdots X_{\mov+1}=0
\end{equation}
It is left to the reader to analyze the hypothesis of Theorem \ref{maintheo} in this case and
get results similar to Theorem \ref{maintheo2}.

The degree of the components of the Noether-Lefschetz locus in the case of K3 surfaces is related 
to Gromov-Witten theory, see for instance \cite{MaulikPandharipande}. It would be interesting
to know whether  such relations can be generalized beyond K3 surfaces. 
The degree of the component of the Hodge locus
in Theorem \ref{maintheo2} for $\mov=4$  can be
computed explicitly, see for instance Vainsencher's article \cite{Vainsencher}.

For the entire collection of problems, computations and forthcoming articles related 
to Gauss-Manin connection in disguise
the reader is referred to  the author's web page. The case of Calabi-Yau varieties and its application
in Topological String Theory can be found in \cite{HosseinMurad}.


\def\cprime{$'$} \def\cprime{$'$} \def\cprime{$'$}

\end{document}